\documentclass[12pt]{article}

\usepackage{url}
\usepackage{mathtools}
\usepackage{amssymb}
\usepackage{amsthm}
\usepackage{empheq}
\usepackage{latexsym}
\usepackage{enumitem}
\usepackage{dsfont}
\usepackage{appendix}
\usepackage{color} 
\usepackage[unicode]{hyperref}
\usepackage[utf8]{inputenc}
\usepackage[T1]{fontenc}
\usepackage{geometry}
\usepackage{multirow}
\usepackage{todonotes}

\DeclareFontFamily{U}{txsyc}{}
\DeclareFontShape{U}{txsyc}{m}{n}{
   <-> txsyc%
}{}
\DeclareFontShape{U}{txsyc}{bx}{n}{
   <-> txbsyc%
}{}
\DeclareFontShape{U}{txsyc}{l}{n}{<->ssub * txsyc/m/n}{}
\DeclareFontShape{U}{txsyc}{b}{n}{<->ssub * txsyc/bx/n}{}
\DeclareSymbolFont{symbolsC}{U}{txsyc}{m}{n}
\SetSymbolFont{symbolsC}{bold}{U}{txsyc}{bx}{n}
\DeclareFontSubstitution{U}{txsyc}{m}{n}
\DeclareMathSymbol{\df}{\mathrel}{symbolsC}{"42}
\DeclareMathSymbol{\fd}{\mathrel}{symbolsC}{"43}
\DeclareMathSymbol{\lJoin}{\mathrel}{symbolsC}{"58}
\DeclareMathSymbol{\rJoin}{\mathrel}{symbolsC}{"59}

\newcommand{\cB}{{\cal B}}
\newcommand{\cC}{{\cal C}}
\newcommand{\cD}{{\cal D}}

\newcommand{\cF}{{\cal F}}
\newcommand{\cG}{{\cal G}}

\newcommand{\cK}{{\cal K}}
\newcommand{\cL}{{\cal L}}
\newcommand{\cN}{{\cal N}}

\newcommand{\cS}{{\cal S}}
\newcommand{\cT}{{\cal T}}

\newcommand{\cV}{{\cal V}}

\newcommand{\DD}{\mathbb{D}}
\newcommand{\EE}{\mathbb{E}}

\newcommand{\NN}{\mathbb{N}}
\newcommand{\PP}{\mathbb{P}}

\newcommand{\RR}{\mathbb{R}}

\newcommand{\ZZ}{\mathbb{Z}}

\newcommand{\iy}{\infty}

\newcommand{\lt}{\left}
\newcommand{\me}{\medskip}

\newcommand{\ri}{\rightarrow}
\newcommand{\rt}{\right}
\newcommand{\sm}{\smallskip}
\newcommand{\tr}{\triangle}

\newcommand{\wi}{\widetilde}

\newcommand{\fo}{\forall\ }

\newcommand{\Id}{\mathrm{Id}}

\newcommand{\st}{\,:\,}

\newcommand{\un}{\mathds{1}}

\newcommand{\bq}{\begin{eqnarray*}}
\newcommand{\bqn}[1]{\begin{eqnarray}\label{#1}}
\newcommand{\eq}{\end{eqnarray*}}
\newcommand{\eqn}{\end{eqnarray}}

\newcommand{\ttsim}{\raise.17ex\hbox{$\scriptstyle\mathtt{\sim}$}}

\newtheorem{pro}{Proposition} 

\newtheorem{lem}[pro]{Lemma}
\newtheorem{theo}[pro]{Theorem}
\newtheorem{rem}[pro]{Remark}
\newtheorem{exa}[pro]{Example}

\definecolor{red}{rgb}{0.7,0.15,0.15}
\definecolor{green}{rgb}{0,0.5,0}
\definecolor{blue}{rgb}{0,0,0.7}
\hypersetup{colorlinks, linkcolor={red},citecolor={green}, urlcolor={blue}}
			
\makeatletter \@addtoreset{equation}{section}

\DeclareUnicodeCharacter{014D}{\=o}
\def \E{\mathbb{E}}

\def \P{\mathbb{P}}

\def \R{\mathbb{R}}

\def \Z{\mathbb{Z}}

\def\cB{{\cal B}}
\def\cC{{\cal C}}
\def\cD{{\cal D}}

\def\cF{{\cal F}}
\def\cG{{\cal G}}

\def\cK{{\cal K}}
\def\cL{{\cal L}}

\def\cN{{\cal N}}

\def\cS{{\cal S}}
\def\cT{{\cal T}}

\def\cV{{\cal V}}

\def \indic{1\!\!1}

\newcommand{\bRp}{{\bar\RR}_+}
 \newcommand{\bFp}{{\bar\cF}_+}

\author{Laurent Miclo{\footnote{CNRS and Toulouse School of Economics. }}  \, and St\'ephane Villeneuve{\footnote{Toulouse School of Economics. }}}

\title{On a Monotone Dynamic Approach to Optimal Stopping Problems for Continuous-Time Markov Chains }

\date{\today}

\begin{document}

\maketitle

\begin{abstract}
This paper is concerned with the solution of the optimal stopping problem associated to the valuation of Perpetual American options driven by continuous time Markov chains. 
We introduce a new dynamic approach for the numerical pricing of this type of American options where the main idea is to build a monotone sequence of almost excessive functions that are associated to hitting times 
of explicit sets. Under minimal assumptions about the payoff and the Markov chain, we prove that the value function of an American option is characterized by the limit of this monotone sequence. 

\vspace{0.5cm}

\end{abstract}

\section{Introduction}
Optimal stopping problems have received a lot of attention in the literature on stochastic control since the seminal work of Wald \cite{Wald:45} about sequential analysis 
while the most recent application of optimal stopping problems have emerged from mathematical finance with the valuation of American options and the theory of real options, see e.g. \cite{Myneni:92} and \cite{DixitPindyck:94}. The first general result of optimal stopping theory for stochastic processes was obtained in discrete time by Snell \cite{snell53} who characterized the value function of an optimal stopping problem as the least excessive function that is a majorant of the reward. For a survey of optimal stopping theory for Markov processes, see the book by Shiryaev  \cite{shiryaev78}. Theoretical and numerical aspects of the valuation of American options have been the subject of numerous articles in many different models including discrete-time Markov chains  (see e.g. \cite{CRR:79},\cite{Lamberton:98}), time-homogenous diffusions (see e.g. \cite{DayanikKaratzas:03}) and Lévy processes (see e.g. \cite{Mordecki:02}) . Following the recent study by Eriksson and Pistorius \cite{erikssonpistorius15}, this paper is concerned with optimal stopping problems in the setting of a continuous-time Markov chain. This class of processes, which contains the classic birth-death process, have recently been introduced in finance to model the state of the order book, see \cite{aj:11}. Assuming a uniform integrability condition for the payoff function,  Eriksson and Pistorius \cite{ErikssonPistorius:15} have shown that the value of an optimal stopping problem for a continuous-time Markov chain can be characterized as the unique solution to a system of variational inequalities. Furthermore, when the state space of the underlying Markov chain is a subset of $\R$ and when the stopping region is assumed to be an interval, their paper provides an algorithm to compute the value function.\\  
Our approach is  different and relies on a monotone recursive construction of both the value function and the stopping region along a sequence of almost excessive functions build along the hitting times of explicit sets. Using the Snell characterization  of the value function as the smallest excessive majorant of the payoff function has been already the idea of the two papers \cite{HelmesStockbridge:10} and \cite{christensen:14} in the one-dimensional diffusion case where upper bounds of the value function are build using linear programming.  The main advantage of the monotone approach developed here, is that it converges to the value  with minimal assumptions about the continuous-time Markov chain and the payoff function. In particular, we abandon the uniform integrability condition while the state space is not necessary a subset of the set of real numbers. Such an approach gives a constructive method of finding the value function and seems to be designed for computational methods. It is fair to notice however, that this procedure may give the exact value of the value function only after infinite number of steps. A practical exception is given when considering the case of Markov chains with finite number of states where the resulting algorithm resembles the elimination algorithm proposed in \cite{sonin} and thus converges in a finite number of steps.

\section{Formulation of the problem}
On a countable state space $V$ endowed with the discrete topology, we consider a Markov generator $\cL\df (L(x,y))_{x,y \in V}$, that is an infinite matrix whose entries are real numbers satisfying
\begin{align*}
\forall x \neq y \in V,&\quad L(x,y) \ge 0 \\
\forall x \in V,&\quad L(x,x)=-\sum_{y\neq x} L(x,y)
\end{align*}
We define $L(x)=-L(x,x)$ and assume that $L(x) <+\infty$ for every $x \in V$.

For any probability measure $m$ on $V$, let us  associate to $L$ a  Markov process $X\df(X_t)_{t\geq 0}$  defined on some probability space $(\Omega,\cG,\P)$ whose initial distribution is $m$.
First we set $\sigma_0\df0$ and $X_0$ is sampled according to $m$. Then we consider an exponential random variable $\sigma_1$ of parameter $L(X_0)\df -L(X_0,X_0)$.
If $L(X_0)=0$, we have a.s.\ $\sigma_1=+\iy$ and we take $X_t\df X_0$ for all $t> 0$, as well as $\sigma_n\df+\iy$ for all $n\in\NN, n\geq 2$.
If $L(X_0)>0$, we take $X_t\df X_0$ for all $t\in(0,\sigma_1)$ and we sample $X_{\sigma_1}$ on $V\setminus\{X_0\}$ according to the 
probability distribution $L(X_0,.)/L(X_0)$.
Next, still in the case where $\sigma_1<+\iy$, we sample an inter-time $\sigma_2-\sigma_1$ as an exponential distribution of parameter $L(X_{\sigma_1})$.
If $L(X_{\sigma_1})=0$, we have a.s.\ $\sigma_2=+\iy$ and we take $X_t\df X_{\sigma_1}$ for all $t\in [\sigma_1,+\iy)$, as well as $\sigma_n\df+\iy$ for all $n\in\NN, n\geq 3$.
If $L(X_{\sigma_1})>0$, we take $X_t\df X_{\sigma_1}$ for all $t\in [\sigma_1,\sigma_2)$ and we sample $X_{\sigma_2}$ on $V\setminus\{X_{\sigma_1}\}$ according to the 
probability distribution $(L(X_{\sigma_1},.)/L(X_{\sigma_1}))_{x\in V\setminus\{X_{\sigma_1}\}}$.
We keep on following the same procedure, where
 all the ingredients are independent, except for the explicitly mentioned dependences.
\par
In particular, we get a non-decreasing family $(\sigma_n)_{n\in\ZZ_+}$ of jump times taking values in $\bRp\df\RR_+\sqcup\{+\iy\}$.
Denote the corresponding exploding time
\bq\sigma_{\iy}&\df&\lim_{n\ri\iy} \sigma_n\ \in\ \bRp\eq
When $\sigma_\iy<+\iy$, we must still define $X_t$ for $t\geq \sigma_\iy$.
So introduce $\tr$ a cemetery point not belonging to $V$ and denote $\bar V\df V\sqcup\{\tr\}$. $\bar V$ is seen as the Alexandrov compactification of $V$.
We take $X_t\df \tr$ for all $t\geq \sigma_\iy$ to get  a $\bar V$-valued Markov process $X$.
Let $(\cG_t)_{t \ge 0}$ be the completed right-continuous filtration generated by  $X\df(X_t)_{t\geq 0}$ and
let $\cF$ (resp.$\bFp$) be the set of functions defined on $V$  taking values in $\RR_+$ (resp. $\bRp\df\RR_+\sqcup\{+\iy\}$).
The generator $L$  acts on $\cF$ via
\bq
\fo f\in\cF,\,\fo x\in V,\qquad \cL[f](x)&\df& \sum_{y\in V} L(x,y) f(y)\\
&=&
\sum_{y\in V\setminus\{x\}} L(x,y) (f(y)-f(x)). \eq
We would like to extend this action on $\bFp$, but since its elements are allowed to take the value $+\iy$, it leads to artificial conventions
such as $(+\iy)-(+\iy)=0$. The only reasonable convention  is $0\times (+\iy)=0$, so let us introduce $\cK$,
the infinite matrix whose diagonal entries are zero and which is coinciding with $\cL$ outside the diagonal.
Its interest is that $\cK$ acts obviously on $\bFp$ through
\bqn{K}
\fo f\in\bFp,\,\fo x\in V,\qquad \cK[f](x)&\df& 
\sum_{y\in V\setminus\{x\}} L(x,y) f(y)\ \in\ \bRp. \eqn

In this paper, we will consider an optimal stopping problem with payoff $e^{-rt}\phi(X_t)$, where $\phi \in \bFp$ and $r>0$, given by
\begin{equation}\label{defOSP}
u(x) \df\sup_{\tau \in \cT} \E_x [e^{-r\tau}\phi(X_\tau)],
\end{equation} 
where $\cT$ is a set of $\cG_t$-adapted stopping times and where the $x$ in index of the expectation indicates that $X$ starts from $x\in V$.  A stopping time $\tau^*$ is said to be optimal for $u$ if
\bq 
u(x)&=&\E_x(e^{-r{\tau^*}}\phi(X_{\tau^*})).
\eq 
Observe that with our convention, we have
$
e^{-r\tau}\phi(X_\tau)=0
$ on the set $\{\tau =+\infty\}$.\par
There are two questions to be solved in connection with Definition  \eqref{defOSP}. The first question is to value the  function $u$ while the second  is to find 
an optimal stopping time $\tau^*$. Note that optimal stopping times may not exist (see \cite{shiryaev78} Example 5 p.61) .
According to the general optimal stopping theory, an optimal stopping time, if it exists, is related to the set
\begin{equation}\label{stoppingset}
D=\{x \in V:\, u(x)=\phi(x) \}
\end{equation} 
called the stopping region. In particular, when $\phi$ satisfies the uniform integrability condition 
\bq 
\EE_x\left[ \sup_{t \ge 0} e^{-rt}\phi(X_t) \right]& <&\infty,
\eq 
the stopping time $\tau_D=\inf\{ t \ge 0,\, X_t \in D\}$ is optimal if  for all $x \in V, \P_x( \tau_D<\infty)=1$ (see Shiryaev \cite{shiryaev78}  Theorem 4 p.52).\\
The main objective of this paper is to provide a recursive construction of both the value function $u$ and the stopping region $D$ without assuming the uniform integrability condition. However,
to present the idea of the monotone dynamic approach developed in this paper, Section 3 first consider the case of a finite state space $V$ for which the uniform integrability condition is obviously satisfied. 
Section 4 is devoted to the general case. Section 5 revisits an example of optimal stopping with random intervention times. 
\section{Finite state space}

On a finite set $V$, the payoff function is bounded and thus the value function $u$ defined by \eqref{defOSP} is well-defined for every $x \in V$. Moreover,  it is well-known (see \cite{shiryaev78},  Theorem 3) that the value function $u$  is the minimal $r$-excessive function which dominates $\phi$. Recall that a function $f$ is $r$-excessive if $0 \ge \cL[f]-rf$. Moreover, on the set $\{ u > \phi \}$, $u$ satisfies
$
\cL[u](x)-ru(x)=0
$.
Because of the finiteness of $V$, the process
\bq 
e^{-rt}f(X_t)-f(x)-\int_0^t e^{-rs}\left( \cL[f]-rf\right)(X_s)\,ds
\eq 
is a $\cG_t$-martingale under $\P_x$ for every function $f$ defined on $V$ and every $x \in V$ which yields by taking expectations, the so-called Dynkin’s formula.\\
We first establish some properties of the stopping region $\cD$. Let us introduce the set \bq \cD_1&\df&\{x \in V,\,\cL[\phi](x)-r\phi(x) \le 0\}\eq  and assume that $\phi(x_0)>0$ for some $x_0 \in V$. We
recall that a Markov process $X$ is said to be irreducible if for all $x,y \in V\times V,\, \P_x(T_y<+\infty)>0$ where
\bq 
T_y&\df&\inf\{ t \ge 0,\, X_t=y\}.
\eq 

\begin{lem}
We have the inclusion $\cD \subset \cD_1$ and when we assume furthermore that $X$ is irreducible, we have $\cD \subset \{x \in V, \phi(x)>0\}$.
\end{lem}
\begin{proof} 
Because $u$ is $r$-excessive, we have for all $x \in \cD$,
\begin{align*}
0& \ge \cL[u](x)-ru(x) \\
&=\sum_{{y \neq x}} L(x,y)u(y)-(r+L(x)) \phi(x) \quad\hbox{ because }x \in \cD\\
&\ge \sum_{y \neq x} L(x,y)\phi(y)-(r+L(x))\phi(x) \quad\hbox{ because } u \ge \phi\\
&=\cL[\phi](x)-r\phi (x). 
\end{align*}

Therefore, $x \in \cD_1$.\\

For the second inclusion, let $T_{x_0}$ be the first time $X$ hits $x_0$.  We have for all $x \in V$, and every $t \ge 0$,
 \bq
 u(x)&\ge& \EE_x[e^{-r (T_{x_0}\wedge t)}\phi(X_{T_{x_0}\wedge t})]\\
&=& \phi(x_0)\E_x[e^{-rT_{x_0}} \indic_{T_{x_0} \le t}]+\EE_x[e^{-r t}\phi(X_{ t})\indic_{T_{x_0} \ge t}]
\eq 
Letting $t$ tend to $+\infty$, we obtain because $\phi$ is bounded on the finite sate space $V$
\bq 
u(x)& \ge& \phi(x_0)\E_x[e^{-rT_{x_0}} \indic_{T_{x_0} <+\infty}]\ >\ 0
\eq 
where the last strict inequality follows from the fact that $X$ is irreducible. 

\end{proof} 

Now, we introduce $u_1$  as the value associated to the stopping strategy {\it Stop the first time $X$ enters in $\cD_1$}. Formally, let us define
\bq 
\tau_1&\df&\inf\{ t \ge 0:\, X_t \in \cD_1\} 
\eq  
and
\bq 
u_1(x)&\df&\E_x[e^{-r\tau_1}\phi(X_{\tau_1})\indic_{\tau_{1} <+\infty}]
\eq 
Clearly $u\ge u_1$ by Definition \eqref{defOSP}. Moreover, we have $u_1=\phi$ on $\cD_1$.

\begin{lem} \label{step0}
We have  
\begin{itemize}
\item $\forall x \notin \cD_1$, $u_1(x)>\phi(x)$ and $\cL[u_1](x)-ru_1(x)= 0$. 
\item $\forall x \in \cD$, $\cL[u_1](x)-ru_1(x) \le 0$.
\end{itemize}
\end{lem}

\begin{proof}
Let $x \notin \cD_1$. Applying the Optional Sampling theorem to the bounded martingale
\bq 
M_t&=&e^{-rt}\phi(X_t)-\phi(x)-\int_0^t e^{-rs}\left( \cL[\phi]-r \phi\right)(X_s)\,ds,
\eq 
we have,
\begin{align*}
u_1(x)&=\E_x[e^{-r\tau_1}\phi(X_{\tau_1})]\\
&=\phi(x)+\E_x\left[\int_0^{\tau_1} e^{-rs}(L[\phi](X_s)-r\phi(X_s))\,ds\right]\\
&>\phi(x),
\end{align*}
because $L[\phi](y)-r\phi(y) > 0$ for $y \notin \cD_1$. Moreover, for $x \notin \cD_1$, $\tau_1 \ge \sigma_1$ almost surely. Thus, the Strong Markov property yields
\begin{align*}
u_1(x)&=\E_x[e^{-r\tau_1}\phi(X_{\tau_1})]\\
&=\E_x[e^{-r\sigma_1}u_1(X_{\sigma_1})]\\
&=\frac{\cL[u_1](x)+L(x)u_1(x)}{r+L(x)},
\end{align*}
from which we deduce $\cL[u_1](x)-ru_1(x)= 0$. \\

Because $u$ is $r$-excessive, we have for all $x \in \cD$,
\begin{align*}
0& \ge \cL[u](x)-ru(x) \\
&=\sum_{y \in V} L(x,y)u(y)-r\phi(x) \quad\hbox{ because }x \in \cD\\
&\ge \sum_{y \neq x} L(x,y)u_1(y)-(r+L(x))\phi(x)\\
&=\cL[u_1](x)-ru_1(x) \quad\hbox{ because } \cD \subset \cD_1.
\end{align*}
\end{proof}

To start the recursive construction, we introduce the set 
\bq 
\cD_2&\df&\{x \in \cD_1,\, \cL[u_1](x)-ru_1(x)\le 0 \}
\eq 
and the function
\bq 
u_2(x)&\df&\E_x[e^{-r\tau_2}\phi(X_{\tau_2})\indic_{\tau_2<+\infty}]
\eq 
where
\bq 
\tau_2&\df&\inf\{ t \ge 0:\, X_t \in \cD_2\} 
\eq 
Observe that if $\cD_2=\cD_1$, $u_1$ is a $r$-excessive majorant of $\phi$ and therefore $u_1\ge u$. Because the reverse inequality holds by definition, the procedure stops.

By induction, we shall define a sequence $(u_n,\cD_n)$ for $n \in \Z_+$  starting from $(u_1,\cD_1)$ by
\bq 
\cD_{n+1}&\df&\{x \in \cD_n,\, \cL[u_n](x)-ru_n(x)\le 0 \}
\eq 
and
\bq 
u_{n+1}(x)&\df&\E_x[e^{-r\tau_{n+1}}\phi(X_{\tau_{n+1}})\indic_{\tau_{n+1}<+\infty}]
\eq 
where
\bq 
\tau_{n+1}&\df&\inf\{ t \ge 0: \, X_t \in \cD_{n+1}\} .
\eq

Next lemma proves a key monotonicity result.
\begin{lem}\label{mono_finite}
We have $u_{n+1} \ge u_n$ and $\forall x \notin \cD_{n+1}$, $u_{n+1}(x)> \phi(x)$.
\end{lem}
\begin{proof}

To start the induction, we assume using Lemma \ref{step0} that $u_n$ satisfies
\begin{align}
\forall x&\in V\setminus D_n, \cL[ u_n](x)-ru_n(x)=0 \hbox{ and } u_{n}(x)>\phi(x)\\
\forall x &\in D_n, \quad  u_n(x)=\phi(x).
\end{align}

For $x \in \cD_{n+1} \subset \cD_{n}$, we have $u_{n+1}(x)=\phi(x)=u_{n}(x)$. On the other hand, for $x \notin \cD_{n+1}$, we have
\begin{align*}
u_{n+1}(x)&=\E_x[e^{-r\tau_{n+1}}\phi(X_{\tau_{n+1}})\indic_{\tau_{n+1}<+\infty}]\\
&=\E_x[e^{-r\tau_{n+1}}u_{n}(X_{\tau_{n}})\indic_{\tau_{n+1}<+\infty}]\quad \hbox{ because $\cD_{n+1} \subset \cD_{n}$}\\
&=u_{n}(x)+\E_x\left[\int_0^{\tau_{n+1}} e^{-rs}(\cL[u_{n}](X_s)-ru_{n}(X_s))\,ds\right]\\
&\ge u_{n}(x),
\end{align*}
because $\cL[u_{n}]-u_{n} \ge 0$ outside $\cD_{n+1}$.\\
Let $x \notin \cD_{n+1}$. If  $x \notin \cD_{n}$, we have $u_{n}(x)>\phi(x)$  and thus $u_{n+1}(x)>\phi(x)$. Now, let $x \in \cD_{n} \cap \cD_{n+1}^c$ and let us define
\bq 
\hat \tau&\df&\inf\{ t \ge 0,\, X_t \notin \cD_{n} \cap \cD_{n+1}^c\}.
\eq 
Clearly, $\hat \tau \le \tau_{n+1}$. Therefore by the Strong Markov property,
\begin{align*}
u_{n+1}(x)&=\E_x[e^{-r \hat \tau}u_{n+1}(X_{\hat\tau})\indic_{\hat\tau<+\infty}]\\
&\ge\E_x[e^{-r\hat \tau}u_{n}(X_{\hat\tau})\indic_{\hat\tau<+\infty}]\\
&> u_{n}(x),
\end{align*}
because $\cL[u_{n}]-u_{n} > 0$ on the set $\cD_{n} \cap \cD_{n+1}^c$.
\end{proof}

According to Lemma \ref{mono_finite}, the sequence $(u_n)_n$ is increasing and satisfies $u_n \ge \phi$ with strict inequality outside $\cD_n$, while by construction, the sequence $(\cD_n)_n$ is decreasing.
It follows that we can define a function $u_\iy$ on $V$ by
\bq
 u_\iy(x)&=&\lim_{n\ri\iy}u_n(x)
\eq
and a set by
\bq
\cD_\iy&\df&\bigcap_{n\in \ZZ_+} \cD_n. 
\eq
We are in a position to state our first result. 
\begin{theo}\label{main_finite}
We have $u_\iy=u$ and $\cD_\iy=D$.
\end{theo}
\begin{proof} By definition, $u \ge u_n$ for every $n \in \ZZ_+$ and thus passing to the limit, we have $u \ge u_\iy$. To show the reverse inequality,
we first notice that for every $n \in \ZZ_+$, we have $u_n \ge \phi$ and thus $u_\iy \ge \phi$.\\
If $x \in \cD_\iy$ then $x \in \cD_{n+1}$ for every $n \in \ZZ_+$ and thus $\cL[u_n](x)-ru_n(x)\le 0$  for every $n \in \ZZ_+$. Passing to the limit, we obtain
\bq 
\cL[u_\iy](x)-ru_\iy(x) &\le& 0 \qquad \forall x \in \cD_\iy.
\eq 
If $x \notin \cD_\iy$ then there is some $n_0$ such that $x \notin \cD_n$ for $n \ge n_0$. Thus, for such a $n \ge n_0$, we have
\bq 
\cL[u_n](x)-ru_n(x)&=&0.
\eq 
Passing to the limit, we obtain 
\bq 
\cL[u_\iy](x)-ru_\iy(x) &=& 0 \qquad \forall x \notin \cD_\iy.
\eq 
To conclude, we observe that because for every $x \in V$, we have $ \cL[u_\iy](x)-ru_\iy(x) \le 0 $, we have for every stopping time $\tau$
\bq 
u_\iy(x) &\ge& \EE[ e^{-r\tau}u_\iy(X_\tau)],
\eq 
from which we deduce that
\begin{equation} \label{minore}
u_\iy(x) \ge \EE[ e^{-r\tau}\phi(X_\tau)],
\end{equation}
because $u_\iy \ge \phi$. Taking the supremum over $\tau$ at the right-hand side of \eqref{minore}, we obtain $u_\iy \ge u$.\\
Equality $u=u_\iy$ implies that $\cD_\infty \subset \cD$. To show the reverse inclusion, let $x \notin \cD_\iy$ which means that $x \notin \cD_n$ for $n$ larger than some $n_0$.
Lemma \ref{mono_finite} yields that $u_n(x)>\phi(x)$ for $n \ge n_0$ and because $u_n$ is increasing, we deduce that $u_\iy(x)>\phi(x)$ for $x \notin \cD_\iy$ which concludes the proof.\end{proof}

\begin{rem}
Because $V$ is finite, the sequence $(u_n)_n$ is constant after some $n_0 \le card(V)$ and therefore the procedure stops after at most $card(V)$ steps. 
\end{rem}
 
\begin{exa} \label{ex1}
Let $(X_t)_{t\geq 0}$ be a birth-death process on the set of integers $V_N=\{-N,\ldots,N\}$ stopped the first time it hits $-N$ or $N$. We define for $x \in V_N \setminus \{-N,N\}$,
\bq 
\left\{
\begin{array}{cc}
L(x,x+1)&=\lambda \ge 0,\\
L(x,x-1)&=\mu \ge 0, \\
L(x)&=\lambda + \mu,
\end{array}
\right.
\eq 
and $L(-N)=L(N)=0$. We define $\phi(x)=\max(x,0)$ as the reward function.\\
 Clearly, $u(-N)=0=\phi(-N)$ and $u(N)=N=\phi(N)$ thus the stopping region contains the extreme points $\{-N, N\}$.
We define
\bq 
\cD_1&\df&\{ x \in V_N,\,\cL[\phi](x)-r\phi(x) \le 0\}.
\eq 
A direct computation shows that $\cL[\phi](x)-r\phi(x)=0$ for $-N+1\le x \le -1$, $\cL[\phi](0)-r\phi(0)=\lambda$ and $\cL[\phi](x)-r\phi(x)=\lambda - \mu -r x$ for $ 1 \le x \le N-1$. Therefore,
\bq 
\cD_1&=&\{-N,-N+1,\ldots,-1\}\cup \{x_1,\ldots,N\},
\eq 
with $x_1= \lceil{ \frac{\lambda-\mu}{r}  \rceil} $, where $\lceil{ x  \rceil}$ is the least integer greater than or equal to $x$. In particular, when $\lambda \le \mu$, we have $x_1=0$ and thus $\cD_1=V_N=\cD$. Assume now that $\lambda > \mu $.
To start the induction, we define 
\bq 
\tau_1&\df&\inf\{ t \ge 0:\, X_t \in \cD_1\} 
\eq  
and
\bq 
u_1(x)&\df&\E_x[e^{-r\tau_1}\phi(X_{\tau_1})\indic_{\tau_{1} <+\infty}]
\eq 
and we construct $u_1$ by solving for $0 \le x \le x_1-1$, the linear equation
\bq 
\lambda u_1(x+1)+\mu u_1(x-1)-(r+\lambda+\mu)u_1(x)&=&0 \hbox{ with } u_1(-1)\ =\ 0 \hbox{ and }u_1(x_1)\ =\ x_1.
\eq 
The function $u_1$ is thus explicit and denoting 
\bq 
\Delta &\df& (r+\lambda+\mu)^2-4\lambda\mu>0, 
\eq 
\bq 
\theta_1\ =\ \frac{r+\lambda+\mu-\sqrt{\Delta}}{2\lambda} \hbox{ and } \theta_2\ =\ \frac{r+\lambda+\mu+\sqrt{\Delta}}{2\lambda},
\eq 
we have
\bq 
u_1(x)&=&x_1\frac{\theta_2^{x+1} - \theta_1^{x+1} }{\theta_2^{x_1+1} - \theta_1^{x_1+1}}.
\eq

Observe that $\theta_2+\theta_1>0$ and thus $\cL[u_1](-1)-ru_1(-1)=\lambda u_1(0)>0$. As a consequence, $-1$ does not belong to the set 
\bq 
\cD_2&=&\{ x \in \cD_1,\, \cL[u_1](x)-ru_1(x) \le 0\}.
\eq 
Therefore, if $\cL[u_1](x_1)-ru_1(x_1)=\mu u_1(x_1-1)+\lambda-(r+\mu)x_1\le 0$, we have
\bq 
\cD_2&=&\{-N,-N+1,\ldots,-2\}\cup \{x_1,\ldots,N\},
\eq 
or, if $\cL[u_1](x_1)-ru_1(x_1)=\mu u_1(x_1-1)+\lambda-(r+\mu)x_1 > 0$
\bq 
\cD_2&=&\{-N,-N+1,\ldots,-2\}\cup \{x_1+1,\ldots,N\}.
\eq 
Following our recursive procedure, after $N$ steps, we shall have eliminated the negative integers and thus obtain
\bq 
\cD_N&=&\{-N\}\cup \{x_N,\ldots,N\}
\eq 
for some $x_1 \le x_N \le N$. Note that for $-N \le x \le x_N$, we have
\bq 
u_N(x)&=&x_N\frac{\theta_2^{x+N} - \theta_1^{x+N} }{\theta_2^{x_N+N} - \theta_1^{x_N+N}}.
\eq 
If $\lambda u_N(x_n+1)+\mu u_N(x_N-1)-(r+\lambda+\mu)u_N(x_N)=\lambda -(r+\mu)x_N+\mu x_N\frac{\theta_2^{x_N-1+N} - \theta_1^{x_N-1+N} }{\theta_2^{x_N+N} - \theta_1^{x_N+N}}
\leq 0$,
the stopping region coincides with $\cD_N$, else we define
\bq 
\cD_{N+1}&=&\{-N\}\cup \{x_{N+1},\ldots,N\}, \hbox{ with } x_{N+1}\ =\ x_N+1
\eq 
and
\bq 
u_{N+1}(x)&=&x_{N+1}\frac{\theta_2^{x+N} - \theta_1^{x+N} }{\theta_2^{x_{N+1}+N} - \theta_1^{x_{N+1}+N}}
\eq 
and we repeat the procedure.
\end{exa}

\section{ General state space} 
\subsection{ Countable State Space}

When considering countable finite state space, Dynkin’s formula that has been used in the proofs of Lemma \ref{step0} and \ref{mono_finite} is not directly available, because nothing prevents the payoff to take arbitrarily large values. Nevertheless, we will adapt the strategy used in the case of a finite state space
to build a monotone dynamic approach of the value function in the case of a countable finite state space.

Hereafter, we set some payoff function $\phi\in\bFp\setminus\{0\}$ and $r>0$.
We will construct a subset $D_\iy\subset V$ and a function $u_\iy\in\bFp$ by the following recursive algorithm.
\par
We begin by taking $D_0\df V$ and $u_0\df \phi$. Next, let us assume that $D_n\subset V$ and $u_n\in\bFp$ have been built for some $n\in\ZZ_+$
such that
\begin{align}
\forall x&\in V\setminus D_n, \quad (r+L(x))u_n(x)=\cK[ u_n](x) \label{Pn} \\
\forall x &\in D_n, \quad  u_n(x)=\phi(x).
\end{align}
Observe that it is trivially true for $n=0$.
Then, we define the subset $D_{n+1}$ as follows
\bqn{Dn+1}
D_{n+1}&\df& \{x\in D_n\st \cK[u_n](x)\leq (r +L(x))u_n(x)\}\eqn
where the inequality is understood in $\bRp$.
\par
Next, we consider the stopping time
\bq \label{taun+1}
\tau_{n+1}&\df&\inf\{t\geq 0\st X_t\in D_{n+1}\}\eq
with the usual convention that $\inf \emptyset =+\iy$.
For $m\in\ZZ_+$, define furthermore the stopping time
\bq
\tau_{n+1}^{(m)}&\df&\sigma_m\wedge\tau_{n+1}\eq
and the function $u_{n+1}^{(m)}\in\bFp$ given by
\bqn{umn}
\fo x\in V,\qquad u_{n+1}^{(m)}(x)&\df& \EE_x[\exp(-r \tau_{n+1}^{(m)})u_n(X_{\tau_{n+1}^{(m)}})].\eqn
\par
\begin{rem}\label{rem1}
The non-negative random variable $\exp(-r \tau_{n+1}^{(m)})u_n(X_{\tau_{n+1}^{(m)}})$ is well-defined,
even if $\tau_{n+1}^{(m)}=+\iy$, since the convention $0\times (+\iy)=0$ imposes that $\exp(-r \tau_{n+1}^{(m)})u_n(X_{\tau_{n+1}^{(m)}})=0$
whatever would be $X_{\tau_{n+1}^{(m)}}$, which is not defined in this case.
The occurrence of $\tau_{n+1}^{(m)}=+\iy$ should be quite exceptional: we have
\bq
\{\tau_{n+1}^{(m)}=+\iy\}&=&\{\tau_{n+1}=+\iy\hbox{ and } L(X_{\tau_{n+1}^{(m)}})=0\}\eq
in particular it never happens if $L(x)>0$ for all $x\in V$, i.e.\ when $\tr$ is the only possible absorbing point for $X$.
\end{rem}
Our first result shows that the sequence $(u^{(m)}_{n+1})_{m\in\ZZ_+ }$ is non-decreasing.
\begin{lem}\label{lem1}
We have
\bq
\fo m\in\ZZ_+,\,\fo x\in V,\qquad 
u_{n+1}^{(m)}(x)&\leq &u_{n+1}^{(m+1)}(x)\eq
\end{lem}
\begin{proof}
We first compute 
\bq
u_{n+1}^{(m+1)}(x)&\df& \EE_x[\exp(-r \tau_{n+1}^{(m+1)})u_n(X_{\tau_{n+1}^{(m+1)}})]\\
&=&
\EE_x[\un_{\tau_{n+1}\leq \sigma_m}\exp(-r \tau_{n+1}^{(m+1)})u_n(X_{\tau_{n+1}^{(m+1)}})]+\EE_x[\un_{\tau_{n+1}> \sigma_m}\exp(-r \tau_{n+1}^{(m+1)})u_n(X_{\tau_{n+1}^{(m+1)}})]\eq
Note that on the event $\{\tau_{n+1}\leq \sigma_m\}$, we have that $\tau_{n+1}^{(m+1)}=\tau_{n+1}=\tau_{n+1}^{(m)}$, so the first
term in the above r.h.s.\ is equal to
\bqn{EE1}
\EE_x[\un_{\tau_{n+1}\leq \sigma_m}\exp(-r \tau_{n+1}^{(m+1)})u_n(X_{\tau_{n+1}^{(m+1)}})]&=&\EE_x[\un_{\tau_{n+1}\leq \sigma_m}\exp(-r \tau_{n+1}^{(m)})u_n(X_{\tau_{n+1}^{(m)}})]\eqn
On the event $\{\tau_{n+1}> \sigma_m\}$, we have that $\tau_{n+1}^{(m+1)}=\tau_{n+1}^{(m)}+\sigma_1\circ \theta_{\tau_{n+1}^{(m)}}$,
where $ \theta_{t}$, for $t\geq 0$, is the shift operator by time $t\geq 0$ on the underlying canonical probability space 
$\DD(\RR_+,\bar V)$ of cÃ dlÃ g trajectories.
Using the Strong Markov property of $X$, we get that
\bqn{EE2}
\hskip-10mm\EE_x[\un_{\tau_{n+1}> \sigma_m}\exp(-r \tau_{n+1}^{(m+1)})u_n(X_{\tau_{n+1}^{(m+1)}})]&=&
\EE_x\lt[\un_{\tau_{n+1}> \sigma_m}\exp(-r \tau_{n+1}^{(m)})\EE_{X_{\tau_{n+1}^{(m)}}}[\exp(-r \sigma_1)u_n(X_{\sigma_1})]\rt]
\eqn
For $y\in V$, consider two situations:
\begin{itemize}
\item if $L(y)=0$, we have a.s.\ $\sigma_1=+\iy$ and as in Remark \ref{rem1}, we get
\bq
\EE_y[\exp(-r \sigma_1)u_n(X_{\sigma_1})]&=&0\eq
\item
 if $L(y)>0$, we compute, in $\bRp$,
\bq
\EE_y[\exp(-r \sigma_1)u_n(X_{\sigma_1})]&=&\int_0^{+\iy} \exp(-rs) L(y)\exp(-L(y)s)\sum_{z\in V\setminus\{y\}} \frac{L(y,z)}{L(y)}u_n(z)\\
&=&\int_0^{+\iy} \exp(-rs) \exp(-L(y)s)\sum_{z\in V\setminus\{y\}} L(y,z)u_n(z)\\
&=&\frac1{r+L(y)}\cK[u_n](y)\eq
\end{itemize}
By our conventions, the equality 
\bqn{rhs}
\EE_y[\exp(-r \sigma_1)u_n(X_{\sigma_1})]&=&\frac1{r+L(y)}\cK[u_n](y)
\eqn
is then true for all $y\in V$.
\par
For $y\in D_n$, due to \eqref{Pn}, the r.h.s.\ is equal  to $u_n(y)$.
For $y\in D_n\setminus D_{n+1}$, by definition of $D_{n+1}$ in 
\eqref{Dn+1}, the r.h.s.\ of \eqref{rhs} is bounded below by $u_n(y)$.
It follows that for any 
$y\not\in D_{n+1}$, 
\bq
\EE_y[\exp(-r \sigma_1)u_n(X_{\sigma_1})]&\geq & u_n(y)\eq
\par
On the event $\{\tau_{n+1}> \sigma_m\}$, we have $X_{\tau_{n+1}^{(m)}}\not\in D_{n+1}$ and thus
\bq
\EE_{X_{\tau_{n+1}^{(m)}}}[\exp(-r \sigma_1)u_n(X_{\sigma_1})]&\geq & u_n(X_{\tau_{n+1}^{(m)}})\eq
Coming back to \eqref{EE2}, we deduce
that
\bq
\EE_x[\un_{\tau_{n+1}> \sigma_m}\exp(-r \tau_{n+1}^{(m+1)})u_n(X_{\tau_{n+1}^{(m+1)}})]&\geq & 
\EE_x[\un_{\tau_{n+1}> \sigma_m}\exp(-r \tau_{n+1}^{(m)})u_n(X_{\tau_{n+1}^{(m)}})]\eq
and taking into account  \eqref{EE1}, we conclude that
\bq
u_{n+1}^{(m+1)}(x)&\geq &\EE_x[\un_{\tau_{n+1}\leq  \sigma_m}\exp(-r \tau_{n+1}^{(m)})u_n(X_{\tau_{n+1}^{(m)}})]+\EE_x[\un_{\tau_{n+1}> \sigma_m}\exp(-r \tau_{n+1}^{(m)})u_n(X_{\tau_{n+1}^{(m)}})]\\
&=&\EE_x[\exp(-r \tau_{n+1}^{(m)})u_n(X_{\tau_{n+1}^{(m)}})]\\
&=&u_{n+1}^{(m)}(x)\eq
\end{proof}
\par
The monotonicity property of Lemma \ref{lem1} enables us to define the function $u_{n+1}\in\bFp$ via
\bq
\fo x\in V,\qquad u_{n+1}(x)&\df& \lim_{m\ri\iy} u_{n+1}^{(m)}(x)\eq
ending the iterative construction of the pair $(D_{n+1}, u_{n+1})$ from $(D_{n}, u_{n})$.
It remains to check 
that:
\begin{lem}
The assertion \eqref{Pn} is satisfied with $n$ replaced by $n+1$.
\end{lem}
\begin{proof}
Consider $x\in V\setminus D_{n+1}$ for which $\tau^{(m+1)}_{n+1} \ge \sigma_1$, $\P_x$ a. s..
For the Markov process $X$ starting from $x$, we have for any $m\in\ZZ_+$,
\bq
\tau^{(m+1)}_{n+1}&=&\sigma_1+\tau^{(m)}_{n+1}\circ \sigma_1\eq
The Strong Markov property of $X$ then implies that
\bq
u_{n+1}^{(m+1)}(x)&=&\EE_x\lt[\exp(-r\sigma_1)\EE_{X_{\sigma_1}}[\exp(-r\tau^{(m)}_{n+1})u_n(X_{\tau^{(m)}_{n+1}})]\rt]\\
&=&\EE_x\lt[\exp(-r\sigma_1)u_{n+1}^{(m)}(X_{\sigma_1})\rt]\\
&=&\frac{1}{r+L(x)}\cK[u_{n+1}^{(m)}](x)
\eq
by resorting again to the computations of the proof of Lemma \ref{lem1}.
Monotone convergence insures that
\bq
\lim_{m\ri\iy}\cK[u_{n+1}^{(m)}](x)&=&\cK[u_{n+1}](x)\eq
so we get that for $x\in V\setminus D_{n+1}$,
\bq
(r+L(x))u_{n+1}(x)&=&\cK[u_{n+1}](x)\eq
as wanted.\end{proof}
\par
The sequence $(D_n)_{n\in\ZZ_+}$ is non-increasing by definition, as a consequence we can define
\bq
D_\iy&\df&\bigcap_{n\in\ZZ_+}D_n \eq
\par
From Lemma \ref{lem1}, we deduce that for any $n\in\ZZ_+$,
\bq
\fo x\in V,\qquad u_{n+1}(x)&\geq & u_{n+1}^{(0)}(x)\\
&=&u_n(x)\eq
It follows that we can define the function $u_\iy\in\bFp$ as the non-decreasing limit
\bq
\fo x\in V,\qquad u_\iy(x)&=&\lim_{n\ri\iy}u_n(x)\ \in\ \bRp\eq
\par
\me
The next two propositions establish noticeable properties of the pair $(D_\iy,u_\iy)$:
\begin{pro} \label{excessive}
We have:
\bq
\fo x\in D_\iy,\qquad\lt\{\begin{array}{rcl}
u_\iy(x)&=&\phi(x)\\
\cK[u_\iy](x)&\leq &(r+L(x))u_\iy(x)\end{array}\rt.\\
\fo x\in V\setminus D_\iy,\qquad\lt\{\begin{array}{rcl}
u_\iy(x)&\geq &\phi(x)\\
\cK[u_\iy](x)&= &(r+L(x))u_\iy(x)\end{array}\rt.
\eq
\end{pro}
\begin{proof}
Since $u_0=\phi$, the  fact that $(u_n)_{n\in\ZZ_+}$ is a non-decreasing sequence   implies
that $u_\iy\geq \phi$. To show there is an equality on $D_\iy$, it is sufficient
to show that
\bq
\fo n\in\ZZ_+,\,\fo x\in D_n,\qquad u_n(x)&=&\phi(x)\eq
This is proven by an iterative argument on $n\in\ZZ_+$.
For $n=0$, it corresponds to the equality $u_0=\phi$.
Assume that $u_n=\phi$ on $D_n$, for some $n\in\ZZ_+$.
For $x\in D_{n+1}$, we have $\tau_{n+1}=0$ and thus for 
 any $m\in\ZZ_+$, we get $\tau_{n+1}^{(m)}=0$.
 From \eqref{umn}, we deduce that
 \bq
 \fo x\in D_{n+1},\qquad u^{(m)}_{n+1}\ =\ u_n(x)\ =\ \phi(x)\eq
 Letting $m$ go to infinity, it yields that $u_{n+1}=\phi$ on $D_{n+1}$.\par\sm
 Consider $x\in V\setminus D_\iy$. There exists $N(x)\in\ZZ_+$ such that for any $n\geq N(x)$, we have $x\in V\setminus D_n$.
 Then passing at the limit for large $n$ in \eqref{Pn}, we get, via another use of monotone convergence, that
 \bq
 \fo x\in V\setminus D_{\iy},\qquad (r+L(x))u_\iy(x)=\cK[u_\iy](x)\eq
 For $x\in D_\iy$, we have $x\in D_{n+1}$ for any $n\in \ZZ_+$ and thus from \eqref{Dn+1},
 we have
$ \cK[u_{n}](x)\leq (r+L(x))u_{n}(x)$. Letting $n$ go to infinity, we deduce that
\bq
\fo x\in D_\iy,\qquad  \cK[u_{\iy}](x)\leq (r+L(x))u_{\iy}(x)\eq
\end{proof}
\par\sm

In fact, $u_\iy$ is a strict majorant of $\phi$ on $V\setminus D_\iy$ as proved in the following
\begin{pro}\label{strict}
We have
\bq
\fo x\in V\setminus D_\iy, \qquad u_\iy(x)&>& \phi(x)\eq
It follows that
\bq
D_\iy&=&\{x\in V\st u_\iy(x)=\phi(x)\}
\eq
\end{pro}
\begin{proof}
Consider $x\in V\setminus D_\iy$, there exists a first integer $n\in\ZZ_+$ such that $x\in D_n$ and $x\not\in D_{n+1}$.
From \eqref{Pn} and $x\in V\setminus D_{n+1}$, we deduce that
\bq
\cK[u_{n+1}](x)&= &(r+L(x))u_{n+1}(x) 
\eq
From \eqref{Dn+1} and $x\in V\setminus D_{n+1}$, we get
\bq
\cK[u_{n}](x)&>& (r+L(x))u_{n}(x)\eq
Putting together these two inequalities and the fact that $\cK[u_{n+1}]\geq \cK[u_n]$, we end up with
\bq
(r+L(x))u_{n+1}(x) &>&(r+L(x))u_{n}(x)\eq
which implies that 
\bq
\phi(x)&\leq & u_{n}(x)\\
&<&u_{n+1}(x)\\
&\leq & u_\iy(x)\eq
namely $\phi(x)<u_\iy(x)$.\par
This argument shows that 
\bq
\{x\in V\st u_\iy(x)=\phi(x)\}&\subset & D_\iy\eq
The reverse inclusion is deduced from Proposition \ref{excessive}.
\end{proof}\par
Another formulation of the functions $u_n$, for $n\in\NN$, will be very useful for the
characterization of their limit $u_\iy$.
For $n,m\in\ZZ_+$, let us modify Definition \eqref{umn} to define a function $\wi u_{n+1}^{(m)}$ as
\bqn{umn2}
\fo x\in V,\qquad \wi u_{n+1}^{(m)}(x)&\df& \EE_x\lt[\exp(-r \tau_{n+1}^{(m)})\phi(X_{\tau_{n+1}^{(m)}})\rt]\eqn
A priori there is no monotonicity with respect to $m$, so we define
\bq
\fo x\in V,\qquad \wi u_{n+1}(x)&\df& \liminf_{m\ri\iy} \wi u_{n+1}^{(m)}(x)\eq
\par
A key observation is:
\begin{lem}\label{uwiu}
For any $n\in\NN$, we have 
$\wi u_n=u_n$.
\end{lem}
\begin{proof}
Since for any $n\in\ZZ_+$, we have $u_n\geq \phi$, we get from a direct comparison between \eqref{umn}  and \eqref{umn2}
that for any $m\in\ZZ_+$, $\wi u_{n+1}^{(m)}\leq u_{n+1}^{(m)}$, so  letting $m$ go to infinity, we deduce that
\bqn{wiuu}
\wi u_{n+1}&\leq & u_{n+1}\eqn
The reverse inequality is proven by an iteration over $n$.
\par
More precisely, since $u_0=\phi$, we get by definition that $\wi u_1= u_1$.
\par
Assume that the equality $\wi u_n= u_n$ is true for some $n\in\NN$, and let us show that $\wi u_{n+1}= u_{n+1}$.
For any $m\in\ZZ_+$, we have
\bq
\fo x\in V,\qquad u_{n+1}^{(m)}(x)&=& \EE_x\lt[\exp(-r \tau_{n+1}^{(m)})u_n(X_{\tau_{n+1}^{(m)}})\rt]\\
&=& \EE_x\lt[\exp(-r \tau_{n+1}^{(m)})\wi u_n(X_{\tau_{n+1}^{(m)}})\rt]\\
&=&\EE_x\lt[\exp(-r \tau_{n+1}^{(m)})\liminf_{l\ri\iy}\wi u_n^{(l)}(X_{\tau_{n+1}^{(m)}})\rt]\\
&\leq & \liminf_{l\ri\iy}\EE_x\lt[\exp(-r \tau_{n+1}^{(m)})\wi u_n^{(l)}(X_{\tau_{n+1}^{(m)}})\rt]\eq
where we used Fatou's lemma.
From \eqref{umn2} and the Strong Markov property, we deduce that
\bq
\EE_x\lt[\exp(-r \tau_{n+1}^{(m)})\wi u_n^{(l)}(X_{\tau_{n+1}^{(m)}})\rt]
&=&
\EE_x\lt[\exp(-r \tau_{n+1}^{(m)})\EE_{X_{\tau_{n+1}^{(m)}}}\lt[\exp(-r \tau_{n+1}^{(l)})\phi(X_{\tau_{n+1}^{(l)}})\rt]\rt]\\
&=&\EE_x\lt[\exp(-r \tau_{n+1}^{(m+l)})\phi(X_{\tau_{n+1}^{(m+l)}})\rt]\eq
It follows that
\bq
u_{n+1}^{(m)}(x)&\leq &  \liminf_{l\ri\iy}\EE_x\lt[\exp(-r \tau_{n+1}^{(m+l)})\phi(X_{\tau_{n+1}^{(m+l)}})\rt]\\
&=&\wi u_{n+1}(x)
\eq
It remains to let $m$ go to infinity to get $u_{n+1}\leq \wi u_{n+1}$
and $u_{n+1}= \wi u_{n+1}$, taking into account \eqref{wiuu}.
\end{proof}
\par\sm
Let $\cT$ be the set of $\bRp$-valued stopping times with respect to the filtration generated by $X$.
For $\tau\in \cT$ and $m\in \ZZ_+$, we define
\bq
\tau^{(m)}&\df& \sigma_m\wedge \tau\eq
Extending the observation of Remark \ref{rem1}, it appears that
 for any $m\in\ZZ_+$,
the quantity 
\bqn{um}
\fo x\in V,\qquad u^{(m)}(x)&\df&\sup_{\tau\in \cT}\EE_x\lt[\exp(-r\tau^{(m)})\phi(X_{\tau^{(m)}})\rt]\eqn
is well-defined in $\bRp$.
It is non-decreasing with respect to $m\in\ZZ_+$, since for any $\tau\in\cT$ and any $m\in\ZZ_+$,
$\tau^{(m)}$ can be written as $\wi\tau^{(m+1)}$, with $\wi \tau\df \tau^{(m)}\in\cT$.
Thus we can define a function $\hat u$ by
\bq
\fo x\in V,\qquad  \hat u(x)&\df& \lim_{m\ri \iy} u^{(m)}(x)\eq
By definition of the value function $u$ given by \eqref{defOSP}, we have $u^{(m)}(x) \le u(x)$ for every $m \in \ZZ_+$ and thus $\hat u(x) \le u(x)$ for every $x \in V$. To show the reverse inequality, consider any stopping time $\tau \in \cT$ and apply Fatou Lemma to get
\begin{align*}
\EE_x\lt[ \exp(-r\tau)\phi(X_{\tau})\rt]&=\EE_x\lt[ \liminf_{m\to \iy} \exp(-r\tau^{(m)})\phi(X_{\tau^{(m)}})\rt] \\
&\le \liminf_{m\to \iy} \EE_x\lt[ \exp(-r\tau^{(m)})\phi(X_{\tau^{(m)}})\rt] \\
&\le \liminf_{m\to \iy} u^{(m)}(x)\\
&\le \hat u(x).
\end{align*}
Therefore, the value function $u$ coincides with the limit of the sequence $(u^{(m)})_{m \in \ZZ_+}$.
At this stage, we recall the definition of the stopping region
\bqn{dfD}
D&\df& \{x\in V\st u(x)=\phi(x)\}\eqn

We are in a position to state our main result
\begin{theo}\label{mainresult}
We have
\bq
u_\iy&=&u\\
D_\iy&=&D.
\eq\end{theo}
\begin{proof}
It is sufficient to show that $u_\iy=u$, since $D_\iy=D$ will then follow from Proposition \ref{strict} and \eqref{dfD}.
\par\sm
We begin by proving the inequality $u_\iy\leq u$.
Fix some $x\in V$.
By considering in \eqref{um} the stopping time $\tau\df\tau_{n+1}$ defined in \eqref{taun+1}, we get for any given $m\in \ZZ_+$,
\bq
u^{(m)}(x)&\geq & \EE_x\lt[\exp(-r\tau_{n+1}^{(m)})\phi(X_{\tau_{n+1}^{(m)}})\rt]\\
&=& \wi u_{n+1}^{(m)}(x)\eq
considered in \eqref{umn2}.
Taking Lemma \ref{uwiu} into account, we deduce that
\bq
u^{(m)}(x)&\geq &  u_{n+1}^{(m)}(x)\eq
and letting $m$ go to infinity, we get $u(x)\geq u_{n+1}(x)$.
It remains to let $n$ go to infinity to show that $u(x)\geq u_{\iy}(x)$.\par\sm
To prove the reverse inequality $u_\iy\geq u$, we will show by induction that for every $x \in V$, every $m \in \ZZ_+$ and every $\tau \in \cT$, we have
\bqn{surmartingale}
\EE_x \lt[  \exp(-r(\sigma_m \wedge \tau)u_\iy(X_{\sigma_m \wedge \tau}) \rt] \le u_\iy(x).
\eqn
For $m=1$, we have, because $u_\iy(X_\tau)=u_\iy(x)$ on the set $\{\tau < \sigma_1\}$,
\begin{align*}
\EE_x \lt[  \exp(-r(\sigma_1 \wedge \tau)u_\iy(X_{\sigma_1 \wedge \tau}) \rt]&=\EE_x \lt[  \exp(-r\sigma_1)u_\iy(X_{\sigma_1})\indic_{\sigma_1\le \tau} \rt]+
\EE_x \lt[  \exp(-r \tau)u_\iy(X_{ \tau})\indic_{\tau < \sigma_1} \rt]\\
&=\EE_x \lt[  \exp(-r \sigma_1)u_\iy(X_{\sigma_1})\indic_{\sigma_1\le \tau} \rt]+u_\iy(x)\EE_x \lt[  \exp(-r \tau)\indic_{\tau < \sigma_1} \rt]\\
&=\EE_x \lt[  \exp(-r \sigma_1)u_\iy(X_{\sigma_1}) \rt]+u_\iy(x) \EE\lt[  \exp(-r \tau)\indic_{\sigma_1 > \tau}\rt]\\
&-  \EE\lt[ \exp(-r \sigma_1)u_\iy(X_{\sigma_1})\indic_{\sigma_1 >\tau}\rt].
\end{align*} 
Focusing on the third term, we observe, that on the set $\{\tau < \sigma_1\}$, we have $\sigma_1=\tau+\hat\sigma_1 \circ\theta_\tau$ where $\hat\sigma_1$ is an exponential random
 variable with parameter $L(x)$ independent of $\tau$. Therefore, the Strong Markov property yields
\begin{align*}
 \EE\lt[ \exp(-r \sigma_1)u_\iy(X_{\sigma_1})\indic_{\sigma_1 >\tau}\rt]&=\EE_x\lt[ \exp(-r \tau)\EE_x \lt[  \exp(-r \hat\sigma_1)u_\iy(X_{\hat\sigma_1}) \rt]\indic_{\sigma_1 >\tau}\rt]\\
 &=\frac{\cK[u_\iy](x)}{r+L(x)}\EE_x\lt[  \exp(-r \tau)\indic_{\sigma_1 > \tau}\rt].
\end{align*}
Hence,
\begin{align*}
\EE_x \lt[  \exp(-r(\sigma_1 \wedge \tau)u_\iy(X_{\sigma_1 \wedge \tau}) \rt]&=\frac{\cK[u_\iy](x)}{r+L(x)}\lt( 1- \EE_x\lt[  \exp(-r \tau)\indic_{\sigma_1 > \tau}\rt]\rt)+u_\iy(x) \EE_x\lt[  \exp(-r \tau)\indic_{\sigma_1 > \tau}\rt]\\
&\le u_\iy(x)
\end{align*}
where the last inequality follows from Proposition \ref{excessive}. This proves the assertion for $m=1$. Assume now that for every $x \in V$ and every $\tau \in \cT$, we have
\bq 
\EE_x \lt[  \exp(-r(\sigma_m \wedge \tau)u_\iy(X_{\sigma_m \wedge \tau}) \rt] &\le &u_\iy(x).
\eq 
Observing that $\sigma_{m+1}\wedge \tau=\sigma_{m}\wedge \tau+(\sigma_{1}\wedge \tau)\circ \theta_{\sigma_{m}\wedge \tau}$, we get
\begin{align*}
\EE_x \lt[  \exp(-r(\sigma_{m+1} \wedge \tau)u_\iy(X_{\sigma_{m+1} \wedge \tau}) \rt]&=\EE_x \lt[  \exp(-r(\sigma_m \wedge \tau)\EE_{X_{\sigma_m \wedge \tau}}\lt( \exp(-r(\sigma_1 \wedge \tau)u_\iy(X_{\sigma_1 \wedge \tau})  \rt) \rt]\\
&\le \EE_x \lt[  \exp(-r(\sigma_m \wedge \tau) u_\iy(X_{\sigma_m \wedge \tau} )\rt]\\
&\le u_\iy(x),
\end{align*}
which ends the argument by induction.
To conclude, we take the limit at the right-hand side of inequality \eqref{surmartingale} to obtain $u(x) \le u_\iy(x)$ for every $x \in V$. 
\end{proof}

\begin{rem} \label{nodiscount}
Because we have financial applications in mind, we choose to work directly with payoffs of the form $e^{-r t} \phi(X_t)$. Observe, however, that our methodology applies when $r=0$ pending the assumption
$\phi(X_\tau)=0$ on the set $\{ \tau=+\infty\}$.
\end{rem}

We close this section by giving a very simple example on the countable state space $\ZZ$ with a bounded reward function $\phi$ such that the recursive algorithm does not stop in finite time because it eliminates only one point at each step. 

\begin{exa}
Let $(X_t)_{t\geq 0}$ be a birth-death process with the generator on $\ZZ$
\bq 
\left\{
\begin{array}{cc}
L(x,x+1)&=\lambda \ge 0,\\
L(x,x-1)&=\mu \ge 0, \\
L(x)&=\lambda + \mu,
\end{array}
\right.
\eq 
We define the reward function as 
\bq 
\phi(x)=\left\{
\begin{array}{c}
0 \hbox{ for } x \le 0\\
1 \hbox{ for } x = 1\\
2 \hbox{ for } x \ge 2.
\end{array}
\right.
\eq 
We assume $r=0$ and $\lambda \ge \mu$. Therefore, $\cD_1=\ZZ\setminus \{1\}$. It is easy to show that
\bq 
u_1(1)&=&\frac{2\lambda}{\lambda+\mu}, \hbox{ and thus } \cL[u_1](0)\ =\ \lambda u_1(1)>0.
\eq 
Therefore, $\cD_2=\ZZ\setminus \{1,0\}$. At each step $n \in \ZZ_+$, because $u_n(1-n)>0$, the algorithm will remove only the integer $1-n$ in the set $\cD_n$. Therefore, it will not reach the stopping region $\cD=\{2,3,\ldots\}$ in finite time.
\end{exa} 

\subsection{Measurable state space}\label{mss}

Up to know, we have only considered continuous-time Markov chains with a discrete state space. But, it is not difficult to see that the results of the previous section can be extended to the case where the state space of the Markov chain is a measurable space. More formally,
 we consider on a measurable state space $(V,\cV)$, a non-negative finite kernel $K$. It is a mapping 
\bq
K\st V\times\cV \ni (x,S)&\mapsto & K(x,S)\in\RR_+\eq
such that
\begin{itemize}
\item for any $x\in V$, $K(x,\cdot)$ is a non-negative finite measure on $(V,\cV)$ (because $K(x,V)\in \RR_+$),
\item for any $S\in \cV$, $K(\cdot,S)$ is a non-negative measurable function on $(V,\cV)$.
\end{itemize}
For any probability measure $m$ on $V$, let us  associate to $K$ a  continuous-time Markov process $X\df(X_t)_{t\geq 0}$ whose initial distribution is $m$.
First we set $\sigma_0\df0$ and $X_0$ is sampled according to $m$. Then we consider an exponential random variable $\sigma_1$ of parameter $K(X_0)\df K(X_0,V)$.
If $K(X_0)=0$, we have a.s.\ $\sigma_1=+\iy$ and we take $X_t\df X_0$ for all $t> 0$, as well as $\sigma_n\df+\iy$ for all $n\in\NN, n\geq 2$.
If $K(X_0)>0$, we take $X_t\df X_0$ for all $t\in(0,\sigma_1)$ and we sample $X_{\sigma_1}$ on $V\setminus\{X_0\}$ according to the 
probability distribution $K(X_0,\cdot)/K(X_0)$.
Next, still in the case where $\sigma_1<+\iy$, we sample an inter-time $\sigma_2-\sigma_1$ as an exponential distribution of parameter $K(X_{\sigma_1})$.
If $K(X_{\sigma_1})=0$, we have a.s.\ $\sigma_2=+\iy$ and we take $X_t\df X_{\sigma_1}$ for all $t\in(\sigma_1,+\iy)$, as well as $\sigma_n\df+\iy$ for all $n\in\NN, n\geq 3$.
If $K(X_{\sigma_1})>0$, we take $X_t\df X_0$ for all $t\in(\sigma_1,\sigma_2)$ and we sample $X_{\sigma_2}$ on $V\setminus\{X_{\sigma_1}\}$ according to the 
probability distribution $(K(X_{\sigma_1},x)/K(X_{\sigma_1}))_{x\in V\setminus\{X_{\sigma_1}\}}$.
We keep on following the same procedure, where
 all the ingredients are independent, except for the explicitly mentioned dependences.
\par
In particular, we get a non-decreasing family $(\sigma_n)_{n\in\ZZ_+}$ of jump times taking values in $\bRp\df\RR_+\sqcup\{+\iy\}$.
Denote the corresponding exploding time
\bq\sigma_{\iy}&\df&\lim_{n\ri\iy} \sigma_n\ \in\ \bRp\eq
When $\sigma_\iy<+\iy$, we must still define $X_t$ for $t\geq \sigma_\iy$.
So introduce $\tr$ a cemetery point not belonging to $V$ and denote $\bar V\df V\sqcup\{\tr\}$.
We take $X_t\df \tr$ for all $t\geq \sigma_\iy$ to get  a $\bar V$-valued process $X$.\\ 
Let $\cB$ be the space of bounded and measurable functions from $V$ to $\RR$. For $f \in \cB$, the infinitesimal generator of $X=(X_t)_{t\geq 0}$ is given by
\bq 
\cL[f](x)&=&\int_V f(y)K(x,dy)-K(x)f(x) \ \df\ \cK[f](x)-K(x)f(x).
\eq 
As in Section 4.1, we set some payoff function $\phi\in\bFp\setminus\{0\}$ and $r>0$.
We will construct a subset $D_\iy\subset V$ and a function $u_\iy\in\bFp$ by our recursive algorithm as follows:
\par
We begin by taking $D_0\df V$ and $u_0\df \phi$. Next, let us assume that $D_n\subset V$ and $u_n\in\bFp$ have been built for some $n\in\ZZ_+$
such that
\bq
\forall x\in V\setminus D_n, \quad (r+L(x))u_n(x)&=&\cK[ u_n](x) \\
\forall x \in D_n, \quad  u_n(x)&=&\phi(x).
\eq
Observe that it is trivially true for $n=0$.
Then, we define the subset $D_{n+1}$ as follows
\bq
D_{n+1}&\df& \{x\in D_n\st \cK[u_n](x)\leq (r +K(x))u_n(x)\}\eq
where the inequality is understood in $\bRp$.
\par
Next, we consider the stopping time
\bq 
\tau_{n+1}&\df&\inf\{t\geq 0\st X_t\in D_{n+1}\}\eq
with the usual convention that $\inf \emptyset =+\iy$. It is easy to check that the proofs of Section 4.1. are directly deduced.

\begin{rem} \label{Poissonization}
Our methodology also applies for discrete Markov chains according to the Poissonization  technique that we recall briefly. Consider a Poisson process $N=(N_t)_t$ of intensity $\lambda$ and a discrete Markov chain $(X_n)_{n\in \Z_+}$ with transition matrix or kernel $P$. Assume that  $(X_n)_{n\in \Z_+}$ and $N=(N_t)_t$ are independent. Then, the process 
$$
X_t=\sum_{n=0}^{N_t} X_n
$$
is a continous-time Markov chain with generator $\cL \df \lambda (P-\Id)$.
\end{rem}

\section{ Application: optimal stopping with random intervention times}
We revisit the paper by Dupuis and Wang \cite{DupuisWang:02} where they consider a class of optimal stopping problems that can be only stopped at Poisson jump times. Consider a probability space $(\Omega,\cF\df(\cF_t)_{t\geq 0},\PP)$ satisfying the usual conditions. For $x>0$, let $(S_t^x)_{t\geq 0}$ be a geometric Brownian motion solving the stochastic differential equation
\bq \label{gbm}
\frac{dS_t^x}{S_t^x}&=&b\,dt+\sigma\,dW_t,\qquad S^x_0\ =\ x,
\eq 
where $W=(W_t)_{t\geq 0}$ is a standard $\cF$-Brownian motion and $b$ and $\sigma>0$ are constants. 
When $x=0$, we take $S^0_t=0$ for all times $t\geq 0$.
The probability space is rich enough to carry a $\cF$-Poisson process  $N=(N_t)_{t\geq 0}$ with intensity $\lambda>0$ that is assumed to be independent from $W$. The jump times of the Poisson process are denoted by $T_n$ with $T_0=0$\\
In \cite{DupuisWang:02}, the following optimal stopping problem is considered
\bq 
u(0,x)&=&\sup_{\tau \in \cS_0} \E\left[ e^{-r\tau}(S^x_\tau-K)_+\right],
\eq  
where $r >b$ and 
$\cS_0$ is the set of $\cF$-adapted stopping time $\tau$ for which $\tau(\omega)=T_n(\omega)$ for some $n\in \ZZ_+$.  Similarly to \cite{DupuisWang:02}, let us define $\cG_n=\cF_{T_n}$ and the  $\cG_n$-Markov chain $Z_n=(T_n,S^x_{T_n})$ to have
\bq 
u(0,x)&=&\sup_{N \in \cN_0} \E\left[ \psi(Z_N)\vert Z_0=(0,x)\right], \hbox{ where } \psi(t,x)\ =\ e^{-rt}(x-K)_+
\eq 
and $\cN_0$ is the set of $\cG$-stopping time with values in $\ZZ_+ $. 
To enter the continuous-time framework of the previous sections, we use Remark \ref{Poissonization} with an independent Poisson process  $\wi N=(\wi N_t)_t$ with intensity $1$.
To start our recursive approach, we need to compute the infinitesimal generator $\wi \cL$ of the continuous Markov chain $(\wi Z_t=\sum_{i=0}^{\wi N_t} Z_n)_{t\geq 0}$ with state space $V=\R_+\times\R_+$ in order to define 
$$
\wi \cD_1\df\{(t,x)\in V;\,\wi\cL[\psi](t,x) \le 0\}.
$$
Let $f$ be a bounded and measurable function on $V$. According to Remark \ref{Poissonization}, we have,
\bq 
\wi\cL[f](t,x)&=&\lambda \int_0^{+\infty} \E\left[ f(t+u,S_u^x )\right]e^{-\lambda u}\,du-f(t,x).
\eq 
Because $\psi(t,x)=e^{-rt}\phi(x)$ with $\phi(x)=(x-K)_+$, we have
\bq 
\wi\cL[\psi](t,x)&=&e^{-rt}\left( \lambda R_{r+\lambda}[\phi](x)-\phi(x)\right),
\eq 
where
\bq 
R_{r+\lambda}[\phi](x)&=&\int_0^{+\infty} \E[\phi(S_u^x)]e^{-(r+\lambda)u}\,du
\eq 
is the resolvent of the continuous Markov process $S^x=(S_t^x)_{t\geq 0}$.  Therefore, we have $\wi\cD_1=\RR_+\times\cD_1$ with
\bq
\cD_1&\df&\{x \in \R_+,\, \lambda R_{r+\lambda}[\phi](x)-\phi(x)\le 0\}.
\eq
First, we observe that $\cD_1$ is an interval $[x_1,+\infty[$. Indeed, let us define
\bq 
\eta(x)&\df&\lambda R_{r+\lambda}[\phi](x)-\phi(x).
\eq 
Clearly, $\eta(x)>0$ for $x\le K$. Moreover, for $x>K$,
\bq 
\eta^\prime(x)&=&\left(\lambda\int_0^{+\infty} \partial_x\E[\phi(S_u^x)]e^{-(r+\lambda)u}\,du\right)-\phi'(x).
\eq 
It is well-known that $\partial_x \E[\phi(S_u^x)] \le e^{bu}$ for any $x\geq 0$ and thus, because $r>b$,
\bq 
\eta^\prime(x)\le \frac{\lambda}{r-b+\lambda}-1&<&0,
\eq 
which gives that $\eta$ is a decreasing function on $[K,+\iy)$. It follows that if $\cD_1$ is not empty, then it will be an interval of the form $[x_1,+\iy)$. Now,
\bq 
\eta(x)&=&x\left(\lambda\int_0^{+\infty}\frac{ \E[\phi(S_u^x)]}{x}e^{-(r+\lambda)u}\,du\right)-\phi(x).
\eq 
Because $\displaystyle{\frac{ \E[\phi(S_u^x)]}{x}} \le e^{bu}$, we have
$
\eta(x) \le \left(\frac{\lambda}{r-b+\lambda}-1\right)x+K
$
and therefore, $\eta(x)\le 0$ for $x \ge \left(1+\frac{\lambda}{r-b}\right)K$ which proves that $\cD_1$ is not empty.\par
We will now prove by induction that for every $n \in \NN$, $\wi\cD_n=\RR_+\times\cD_n$ with $\cD_n=[x_n,+\iy)$ and $x_n > K$.
Assume that it is true for some $n\in\NN$. Following our monotone procedure with Remark \ref{nodiscount}, we define
the solution $u_{n+1}\st\RR_+\times\RR_+\ri\RR$ of the equation
\bq
\lt\{\begin{array}{rcll}
\wi\cL [u_{n+1}]&=&0&,\hbox{ on }\wi \cD_{n}^{\mathrm{c}}\\
u_{n+1}&=&\psi &,\hbox{ on }\wi \cD_n
\end{array}\rt.\eq
and define
$$
\wi \cD_{n+1} \df \{ (t,x) \in \wi D_n,\, \wi\cL[ u_{n+1}] \le 0 \}.
$$
Let us check that $\wi\cD_{n+1}=\RR_+\times\cD_{n+1}$ 
 with $\cD_{n+1}=[x_{n+1},+\iy)$ and $x_{n+1} \ge x_n$.\par\sm
To do this, we look for a function of the form
\bqn{unvn}
\fo (t,x)\in\RR_+\times\RR_+,\qquad u_{n+1}(t,x)&=&\exp(-rt)v_{n+1}(x)\eqn
We end up with the following equation on $v_{n+1}$:
\bqn{Rrl}
\lt\{\begin{array}{rcll}
\lambda R_{r+\lambda}[v_{n+1}]-v_{n+1}&=&0&\hbox{, on $ \cD_n^{\mathrm{c}}$}\\
v_{n+1}&=&\phi &\hbox{, on $\cD_n$}\end{array}\rt.\eqn
or equivalently, (see \cite{EthierKurtz:05}, Proposition 2.1 page 10)
\bq
\lt\{\begin{array}{rcll}
\cL^S[v_{n+1}]-r v_{n+1}&=&0&\hbox{, on $ \cD_n^{\mathrm{c}}$}\\
v_{n+1}&=&\phi &\hbox{, on $\cD_n$}\end{array}\rt.\eq
 where $\cL^S$ is the infinitesimal generator of $S^x=(S_t^x)_{t\geq 0}$, that is, acting on any $f\in\cC^2(\RR_+)$ via
\bq
\cL^S[f](x)&=&\frac{\sigma^2x^2}{2}f^{\prime\prime}(x)+bxf^{\prime}(x),
\eq
With this formulation we see that $v_{n+1}$ is given by
\bq
\fo x\in\RR_+,\qquad v_{n+1}(x)&=&\EE_x[\exp(-r\tau_{x_n}^x)\phi(S^x_{\tau_{x_n}})]\eq
where $\tau_{x_n}$ is the first hitting time of $\cD_n=[x_n,+\infty[$ by our induction hypothesis.\par 
By definition, we have
\bq
\wi\cD_{n+1}&\df&\{(t,x)\in\wi\cD_n\st \wi\cL [u_{n+1}](t,x)\leq 0\}\\
&=&\RR_+\times\{x\in\cD_n\st \lambda R_{r+\lambda}[v_{n+1}](x)-v_{n+1}(x)\leq 0\}\\
&=&\RR_+\times\{x\in\cD_n\st \lambda R_{r+\lambda}[v_{n+1}](x)-\phi(x)\leq 0\}
\eq
thus $\wi\cD_{n+1}=\RR_+\times\cD_{n+1}$ where
\bq
\cD_{n+1}&\df&\{x\in\cD_n\st \zeta_{n+1}(x)\leq 0\}\eq
with
\bq
\fo x\geq 0,\qquad \zeta_{n+1}(x)&\df& \lambda R_{r+\lambda}[v_{n+1}](x)-\phi(x)\eq
\par
To prove that $\cD_{n+1}$ is of the form $[x_{n+1},+\iy)$, we begin by showing that
\bqn{zetazeta}
\fo y\geq x\geq x_n, \qquad \zeta_{n+1}(x)=0&\Rightarrow& \zeta_{n+1}(y)\leq 0\eqn
\par
To do so, introduce the hitting time
\bq
\tau_x^y&\df&\inf\{t\geq 0\st S_t^y=x\}\eq
 Recall that the solution of \eqref{gbm} is given
  by
  \bq
  \fo x\in\RR+,\,\fo t\geq 0,\qquad
  S_t^x&=&x\exp\lt(\sigma W_t-\frac{\sigma^2}{2}t +bt\rt)\eq
 It follows that 
 \bq
 \tau_x^y&=&\inf\{t\geq 0\st W_t-\frac{\sigma^2}{2}t +bt=\ln(x/y)\}\eq
In particular $\tau^y_x$  takes the value $+\iy$ with positive probability when  $b>\sigma^2/2$,
but otherwise $\tau^y_x$ is a.s.\ finite.
Nevertheless, taking into account that for any $z\geq x_n$, we have $v_{n+1}(z)=\phi(z)$, 
we can always write for $ y\geq x\geq x_n$:
\bq
\lefteqn{R_{r+\lambda}[v_{n+1}](y)}\\&=&\EE\lt[\int_0^\iy v_{n+1}(S_u^y)\exp(-(r+\lambda)u)\, du\rt]\\
&=&\EE\lt[\int_0^{\tau_{x}^y}v_{n+1}(S_u^y)\exp(-(r+\lambda)u)\, du+\int_{\tau_{x}^y}^\iy v_{n+1}(S_u^y)\exp(-(r+\lambda)u)\, du\rt]\\
&=&\EE\lt[\int_0^{\tau_{x}^y}\phi(S_u^y)\exp(-(r+\lambda)u)\, du\rt]+
\EE\lt[\exp(-(r+\lambda)\tau_{x}^y)\int_{0}^\iy v_{n+1}(S_{\tau_{x}^y+u}^y)\exp(-(r+\lambda)u)\, du\rt]\\
&=&\EE\lt[\int_0^{\tau_{x}^y}\phi(S_u^y)\exp(-(r+\lambda)u)\, du\rt]+
\EE\lt[\exp(-(r+\lambda)\tau_{x}^y)\rt]R_{r+\lambda}[v_{n+1}](x)\eq
where we use the strong Markov property with the stopping time $\tau_{x}^y$.
Reversing the same argument, with $v_{n+1}$ replaced by $\phi$, we deduce that
\bq
R_{r+\lambda}[v_{n+1}](y)&=&\EE\lt[\int_0^{\tau_{x}^y}\phi(S_u^y)\exp(-(r+\lambda)u)\, du\rt]+
\EE\lt[\exp(-(r+\lambda)\tau_{x}^y)\rt]R_{r+\lambda}[\phi](x) \\&&+\EE\lt[\exp(-(r+\lambda)\tau_{x}^y)\rt](R_{r+\lambda}[v_{n+1}](x)-R_{r+\lambda}[\phi](x))\\
&=&R_{r+\lambda}[\phi](y)+\EE\lt[\exp(-(r+\lambda)\tau_{x}^y)\rt](R_{r+\lambda}[v_{n+1}](x)-R_{r+\lambda}[\phi](x))\eq
Thus, we have
\bq
\zeta_{n+1}(y)&=&\lambda R_{r+\lambda}[\phi](y)-\phi(y)+\lambda\EE\lt[\exp(-(r+\lambda)\tau_{x}^y)\rt](R_{r+\lambda}[v_{n+1}](x)-R_{r+\lambda}[\phi](x))\eq
 In the first part of the above proof, to get the existence of 
$x_1$, we have
shown that the mapping $\zeta_1\df R_{r+\lambda}[\phi]-\phi$ is non-increasing on $[x_1,+\iy)\supset [x_n,+\iy)$, and in particular
\bq
\lambda R_{r+\lambda}[\phi](y)-\phi(y)&\leq & \lambda R_{r+\lambda}[\phi](x)-\phi(x)
\\
&=&\zeta_1(x)
\eq
so we get
\bq
\zeta_{n+1}(y)&\leq & \zeta_1(x)\EE\lt[\exp(-(r+\lambda)\tau_{x}^y)\rt](\lambda R_{r+\lambda}[v_{n+1}](x)-\lambda R_{r+\lambda}[\phi](x))\eq
Assume now that $\zeta_{n+1}(x)=0$. It means that
\bq
\lambda R_{r+\lambda}[v_{n+1}](x)&=&\phi(x)\eq
implying that
\bq
\zeta_{n+1}(y)&\leq &\zeta_1(x)+ \EE\lt[\exp(-(r+\lambda)\tau_{x}^y)\rt](\phi(x)-\lambda R_{r+\lambda}[\phi](x))\\
&\leq & \lt(1-\EE\lt[\exp(-(r+\lambda)\tau_{x}^y)\rt]\rt) \zeta_1(x)
\\
&\leq & 0\eq
since $x\geq x_n\geq x_1$ so that $\zeta_1(x)\leq 0$.\par
This proves \eqref{zetazeta} and ends the induction argument.\par
According to Proposition \ref{excessive} and Theorem \ref{mainresult} (more precisely its extension given in Subsection \ref{mss}),  the value function $u$ and the stopping set $\cD=\{x \in \R_+,\,u(x)=\phi(x)\}$  satisfy
$u(t,x)=e^{-rt}v(x)$, where
\bq
v(x)&=&\lambda R_{r+\lambda}[v](x) \hbox{ on } \R_+ \setminus \cD,
\eq
\bq
v&=&\phi \hbox{ on } \cD
\eq
and
\bq
\cD&=&\displaystyle{\bigcap_{n\in \NN}} [x_n,+\infty[.
\eq
The stopping set is an interval $[x^*,+\infty[$ that may be empty if $x^*$ is not finite.
Using again \cite{EthierKurtz:05}, Proposition 2.1, we obtain
\bq
\cL^S[v](x)-rv(x)&=0 \qquad \forall x \in \R_+ \setminus \cD.
\eq
Therefore, the function $w$ given by $w(x)\df \lambda R_{r+\lambda}[v](x)$ for any $x\in(0,+\iy)$ also satisfies
\bq
\cL^S[w](x)-rw(x)&=&0 \qquad \forall x \in \R_+ \setminus \cD.
\eq
Moreover,
\bq
(-\cL^S+(r+\lambda))[w](x)&=&\lambda v(x)\ =\ \lambda \phi(x)\qquad \forall x \in \cD,
\eq
which yields
\bq
\cL^S[w](x)-rw(x)+\lambda (\phi(x)-w(x))&=&0\qquad\forall x \in  \cD.
\eq
This corresponds to the variational inequality (3.4)-(3-9) page 6 solved in \cite{DupuisWang:02}, establishing that $\cD$ is non-empty.


\end{document}